\documentclass[11pt,reqno,oneside]{amsart}

\usepackage{pdfsync}

\usepackage[toc,page]{appendix}
\usepackage{hyperref}

\newtheorem{theorem}{Theorem}[section]
\newtheorem{lemma}[theorem]{Lemma}
\newtheorem{corollary}[theorem]{Corollary}

\theoremstyle{definition}
\newtheorem{defi}[theorem]{Definition}
\newtheorem{remark}[theorem]{Remark}

\def\O{\Omega}
\def\w{\omega}
\def\o{\omega}
\def\P{\mathcal{P}}
\def\d{{\rm d}}

\newcommand{\uu}{{\bf u}}
\newcommand{\vv}{{\bf v}}
\newcommand{\ww}{{\bf w}}
\newcommand{\di}{\mbox{div\,}}
\newcommand{\R}{{\mathbb R}}
\newcommand{\N}{{\mathbb N}}
\newcommand{\bs}{\bigskip}

\begin{document}

\title[Korn inequality]{Weighted Korn inequality on John domains}

\author[F. L\'opez Garc\'\i a]{Fernando L\'opez Garc\'\i a}
\address{University of California Riverside\\ Department of Mathematics\\ 
900 University Ave. Riverside (92521), CA, USA} 
\email{fernando.lopezgarcia@ucr.edu}

\date{\today}

\begin{abstract} We show a weighted version of Korn inequality on bounded euclidean John domains, where the weights are nonnegative powers of the distance to the boundary. In this theorem, we also provide an estimate of the constant involved in the inequality which depends on the power that appears in the weight and a geometric condition that characterizes John domains. The proof uses a local-to-global argument based on a certain decomposition of functions. 

In addition, we prove the solvability in weighted Sobolev spaces of $\di\uu=f$ on the same class of domains. In this case, the weights are nonpositive powers of the distance to the boundary. The constant appearing in this problem is also estimated. 
\end{abstract}



\maketitle

\section{Introduction}
\label{intro}
\setcounter{equation}{0}

Let $\O\subset\R^n$ be a bounded domain for $n\geq 2$ and $1<p<\infty$. The classical Korn inequality states that 
\begin{eqnarray}\label{Korn ineq.}
\|D\uu\|_{L^p(\Omega)^{n\times n}}\leq C \|\varepsilon(\uu)\|_{L^p(\Omega)^{n\times n}}
\end{eqnarray}
for any vector field $\uu$ in the Sobolev space $W^{1,p}(\O)^n$ under appropriate conditions. By $D\uu$ we denote the differential matrix of $\uu$ and by $\varepsilon(\uu)$ its symmetric part. Namely, 
\[\varepsilon_{ij}(\uu)=\frac{1}{2}\left(\frac{\partial u_i}{\partial x_j}+\frac{\partial u_j}{\partial x_i}\right).\]

 Naturally, the constant $C$ depends only on $\O$ and $p$. This inequality plays a fundamental role in the analysis of the linear elasticity equations, where $\uu$ represents a displacement field of an elastic body. The tensor $\varepsilon(\uu)$ 
is called the linearized strain tensor and (\ref{Korn ineq.}) implies the coercivity of the bilinear form associated to the underlying linear equations. 
The two conditions on the vector field considered by Korn in his seminal works \cite{K1,K2}
were: $\uu=0$ on $\partial \O$ (usually called {\it first case}), and $\int_\O \frac{\partial u_i}{\partial x_j}-\frac{\partial u_j}{\partial x_i}=0$ ({\it second case}).  These two conditions remove the non-constant infinitesimal rigid motions (i.e. fields $\uu$ such that the right-hand side of (\ref{Korn ineq.}) vanishes while the left one does not). 

Inequality (\ref{Korn ineq.}) in the first case can be simply proved on any arbitrary domain ${\O}$ by using the divergence theorem (see \cite{F,H}). Moreover, it is known that the optimal constant is equal to $\sqrt{2}$. However, in this work we deal with Korn inequality in the second case, where its validity depends on the geometry of the domain. This inequality has been studied under different assumptions on the domain. For example, it is known that the inequality is valid if $\Omega$ is a star-shaped domains with respect to a ball (see \cite{R}). This class contains the convex domains. The proof in \cite{R} is based on certain integral representations of the vector field $\uu$ in terms of $\varepsilon(\uu)$. Other authors have also studied this inequality on these domains by using different arguments, see \cite{H,KO,T} and references therein. Uniform domains also verify Korn inequality. This result was proved in \cite{DM} by modifying the extension operator given by Peter Jones in \cite{Jo}. The largest known family of domains where $(\ref{Korn ineq.})$ holds is the class of John domains. This class was introduced by Fritz John in \cite{Joh} and named after him by Martio and Sarvas in \cite{MS}. Let us recall the definition of this family. A bounded domain $\Omega\subset\R^n$, with $n\geq 2$, is called a {\it John domain} with parameter $\beta>1$ if there exists a point $x_0\in\O$ such that every $y\in\O$ has a rectifiable curve parameterized by arc length $\gamma:[0,l]\to\Omega$ such that $\gamma(0)=y$, $\gamma(l)=x_0$ and 
\begin{eqnarray*}
\text{dist}(\gamma(t),\partial\O)\geq \frac{1}{\beta}t
\end{eqnarray*}
for all $t\in[0,l]$, where $l$ is the length of $\gamma$. The set of John domains contains the one of Lipschitz domains and some others with very irregular boundaries, such as Kock snowflakes which has a fractal boundary. A version of Korn inequality different from (\ref{Korn ineq.}) on John domains can be found in \cite{ADM}. This result is obtained as a consequence of the main result of the mentioned article which deals with the solvability of $\di\uu=f$ with an appropriate a-priori estimate. In \cite{DRS}, the authors proved (\ref{Korn ineq.}) on John domains where the vector fields belong to a weighted Sobolev space with weights in the Muckenhoupt class $A_p$. More recently, a weighted version of Korn inequality different from the one treated in this article has been shown in \cite{JK}, where the weights are also nonnegative powers of the distance to the boundary. Its proof is based on the validity of a certain improved Poincar\'e inequality published in \cite{Hu} and generalized later in \cite{CW2}. 

In these notes, we are particularly interested in finding an estimate of the constant that appears in the inequality. This problem has been addressed in several articles. For instance in \cite{Du}, the author estimates the constant in (\ref{Korn ineq.}), with $p=2$, in terms of the ratio between the diameter of $\O$ and that of $B$, if $\O$ is a star-shaped domain with respect to a ball $B$. Another recent article dealing with the estimation of the constant in Korn inequality (an other equivalent results) on star-shaped domains is \cite{CD}. This last article considers planar domains. This problem has also been studied in the classical reference \cite{HP}.  However, we could not find in the literature estimates of the constant in Korn inequality when $\O$ is a John domain. 

The main theorem of these notes shows a weighted version of Korn inequality on John domains, where the weight is a nonnegative power of the distance to the boundary. Moreover, we estimate the Korn's constant in terms of the geometric condition introduced in (\ref{Boman tree}). Similar estimates for weighted Poincar\'e inequalities which depend on the eccentricity of a convex domain has been proved in \cite{ChD,CW}; the authors also consider nonnegative powers of the distance to the boundary. 
 
Given a vector field $\uu$ we denote by $\eta(\uu)$ the skew-symmetric part of the differential matrix $D\uu$ of $\uu$. Namely, 
\[\eta_{ij}(\uu)=\frac{1}{2}\left(\frac{\partial u_i}{\partial x_j}-\frac{\partial u_j}{\partial x_i}\right).\]

\begin{theorem}\label{Korn in John} Let $\Omega\subset\R^n$ be a bounded John domain with $n\geq 2$, $1<p<\infty$ and $\beta\in\R_{\geq 0}$. Then, there exists a constant $C$, depending only on $n$, $p$ and $\beta$, such that  
\begin{eqnarray}\label{weighted Korn John}
\left(\int_\Omega |D\uu|^p\rho^{p\beta}\, \d x\right)^{1/p} \leq C\,K^{n+\beta} \left(\int_\Omega |\varepsilon(\uu)|^p\rho^{p\beta}\, \d x\right)^{1/p}
\end{eqnarray}
for all vector field $\uu\in W^{1,p}(\Omega,\rho^\beta)^n$ that satisfies that $\int_{\O}\eta_{ij}(\uu)\,\rho^{\beta p}=0$, for $1\leq i<j\leq n$. The function $\rho(x)$ is the distance to the boundary of $\Omega$ and the constant $K$ is introduced in the geometric condition (\ref{Boman tree}).
\end{theorem}
%Regarding the optimality of the power $n+\beta$ appearing in the estimate, we conjecture its sharpness when $\beta=0$ however we have not been able to find appropriate examples that prove this fact.

Notice that $\rho^\beta$ does not belong to the $A_p$ Muckenhoupt class for a big enough $\beta>0$. Thus, many of the techniques that use the theory of singular integral operators and depend on the continuity of the Hardy-Littlewood maximal operator may not be applicable in this case.

The rest of the paper is organized as follows: in Section \ref{Preliminaries} we introduce some definitions and notations. In Section \ref{Inequalities} we show how certain decompositions of functions can be used to extend the local validity of Korn inequality to the whole domain $\O$.  In this part of the article $\O$ could be any arbitrary bounded domain. Section \ref{Decomposition on general domains} deals with the existence of the required decomposition of functions. In Section \ref{John} we apply the results proved in the previous two sections on John domains to demonstrate the main result of the article.

\section{Definitions and Preliminaries}
\label{Preliminaries}
\setcounter{equation}{0}

Throughtout the paper, $\Omega\subset\R^n$ is a bounded domain with $n\geq 2$, $1<p,q<\infty$ with $\frac{1}{p}+\frac{1}{q}=1$, and $\o:\O\to\R$ is a positive  measurable function such that $\o^p$ is integrable over $\O$. By $\o^p(U)$ we denote $\int_U \o^p$. As usual, $L^p(\Omega,\w)$ denotes the space of Lebesgue measurable functions $u:\Omega\to\R$
equipped with the norm:
\[\|u\|_{L^p(\Omega,\w)}:=\left(\int_\Omega|u(x)|^p\w^p(x)\,\d x\right)^{1/p}.\]
Similarly, we define the weighted Sobolev spaces $W^{1,p}(\Omega,\w)$ as the
space of weakly differentiable functions $u:\Omega\to\R$ with
the norm:
\[\|u\|_{W^{1,p}(\Omega,\w)}:=\left(\int_\Omega|u(x)|^p\w^p(x)\,\d x
+\sum_{i=1}^n\int_\Omega\left|\frac{\partial u(x)}{\partial
x_i}\right|^p\w^p(x)\,\d x\right)^{1/p}.\]
In what follows, $C$ will denote various positive constants which may vary from line to line. We use $C_a$ or $C(a)$ to denote a constant which only depends on $a$.

Let us introduce the decompositions considered in this article.
\begin{defi}\label{Definition decomposition} Given $m\in\N_0$, let $\mathcal{P}_m$ be the space of polynomials of degree no greater than $m$. Moreover, let $\{\O_t\}_{t\in\Gamma}$ be a collection of open sets that satisfies $\Omega=\bigcup_{t\in \Gamma} \Omega_t$. Now, given $g\in L^1(\Omega)$ a function such that $\int g\, \varphi=0$ for all $\varphi\in\P_m$, we say that a collection of functions $\{g_t\}_{t\in\Gamma}$ is a $\P_m$-{\it orthogonal decomposition of $g$} subordinate to $\{\Omega_t\}_{t\in\Gamma}$ if the following three properties are satisfied:
\begin{enumerate}
\item $g=\sum_{t\in \Gamma} g_t.$
\item $supp(g_t)\subset\Omega_t.$
\item $\int_{\Omega_t} g_t\,\varphi=0$, for all $\varphi\in\P_m.$
\end{enumerate}
\end{defi}
We may also refer to this collection of functions by a $\P_m$-{\it decomposition}. 

A {\it covering} $\{\O_t\}_{t\in \Gamma}$ of $\O$ is a countable collection of subdomains of $\O$ that satisfies $\bigcup_t \O_t=\O$ and the following estimate of the overlap:
\begin{eqnarray}\label{overlapping}
\sum_{t\in\Gamma}\chi_{\Omega_t}(x)\leq N \chi_{\Omega}(x),
\end{eqnarray}
for all $x\in\Omega$. This condition is essential in these notes, specifically in Sections \ref{Inequalities} and \ref{John}.

\section{A decomposition and weighted Korn inequality}
\label{Inequalities}
\setcounter{equation}{0}

In this section, we will show that the validity of weighted Korn inequalities on $\O$ (introduced below) can be obtained from the local validity of this inequality if we have an appropriate decomposition of functions in $L^q(\O,\o^{-1})$. No additional assumptions on the domain are required in this section but being bounded. 

Given $U\subseteq\O$, we say that {\it weighted Korn inequality} holds on $U$ if  
\begin{eqnarray}\label{weighted local Korn}
\|D\uu\|_{L^p(U,\o)}\leq C \|\varepsilon(\uu)\|_{L^p(U,\o)}, 
\end{eqnarray}
for any vector field $\uu\in W^{1,p}(U,\o)^n$ that satisfies that $\int_{U}\eta_{ij}(\uu)\,\o^p=0$, for any $1\leq i<j\leq n$. 
There is an equivalent version of inequality (\ref{weighted local Korn}) which says:
\begin{eqnarray}\label{inf weighted local Korn}
\inf_{\varepsilon(\ww)=0}\|D(\vv-\ww)\|_{L^p(U,\o)}\leq C \|\varepsilon(\vv)\|_{L^p(U,\o)}, 
\end{eqnarray}
where the infimum is taken over the kernel of $\varepsilon$ and $\vv$ is an arbitrary vector field in $W^{1,p}(U,\o)^n$. Let us mention that the vector fields that satisfy $\varepsilon(\ww)=0$ are characterized by  
\[\ww(x)=Ax+b,\]
where $A\in\R^{n\times n}$ is a skew-symmetric matrix and $b\in\R^n$. 

The integrability of $\o^p$ will be used several times in this section but it is required in particular to show that  \[L^q(\Omega,\o^{-1})\subset L^1(\O).\] 
Now, given a natural number $m\in\N_0$ and  we denote by  $V_m(\O,\o^{-1})$ (or simply $V_m$) the subspace of $L^q(\O,\o^{-1})$ given by: 
\begin{eqnarray*}
V_{m}:=\{g\in L^q(\Omega,\o^{-1})\,:\, \int g\varphi=0, \forall\varphi\in\P_m,\text{ and }supp(g)\text{ intersects a finite number of }\Omega_t \}.
\end{eqnarray*}
Recall that $\Omega$ is bounded thus the set of polynomial $\P_m$ is contained in $L^\infty(\Omega)$. Then, using that $L^q(\Omega,\o^{-1})\subset L^1(\Omega)$ we have that $V_m$ is well-defined.

\begin{lemma}\label{weighted dense} 
Given $m\in\N_0$ and a covering $\{\O_t\}_{t\in\Gamma}$ of $\O$ such that each $\Omega_t$ intersects a finite number of $\Omega_s$ with $s\in\Gamma$, it follows that the subspace $S_{m}\subset L^q(\Omega,\o^{-1})$ defined by \[S_{m}:=\{g+\o^p\psi\,/\,g\in V_m\text{ and }\psi\in\P_m\}\] is dense in $L^q(\Omega,\o^{-1})$. Moreover, 
$\|g\|_{L^q(\Omega,\o^{-1})}\leq C\|g+\o^p\psi\|_{L^q(\Omega,\o^{-1})}$, where $C$ does not depend on $g$ nor $\psi$. In the particular case when $m=0$ the constant in the previous inequality is equal to $2$.
\end{lemma}

\begin{proof} Let us remark that $\o^p\psi$ belongs to $L^q(\Omega,\o^{-1})$. Indeed, 
\[\|\o^p\psi\|^q_{L^q(\O,\o^{-1})} = \int_\O\psi^q\o^{pq}\o^{-q}  \leq \|\psi^q\|_{L^\infty(\O)}\|\o^p\|_{L^1(\O)},\]
thus $S_m$ is a subspace of $L^q(\O,\o^{-1})$.  

Notice that any function $F$ in $L^q(\O,\o^{-1})$ can be written as $F=h_F+\o^p\psi_F$, where $h_F$ belongs to $L^q(\O,\o^{-1})$ and satisfies that $\int_{\Omega} h_F\varphi=0$ for all $\varphi\in\P_m$ and $\psi_F$ belongs to $\P_m$. This follows for being $\P_m $ a finite dimensional vector space. Thus, the proof is basically reduced to showing existence of an approximation of $h_F$ in $V_m$ (the support of $h_F$ does not necessarily intersect a finite collection of $\O_t$'s). However, let us go back to show the existence of the representation of functions in $L^q(\O,\o^{-1})$ mentioned above. Let us take an orthonormal basis $\{\psi_i\}_{0\leq i\leq M}$ of $\P_m$, where $M$ is the dimension of $\P_m$, with respect to the inner product 
\[\langle \psi,\varphi\rangle_\O=\int_\O \psi(x)\varphi(x) \o^p(x)\,\d x.\] 
Indeed, the basis satisfies that $\int_\O \psi_i\psi_j \o^p=\delta_{ij}$, where $\delta_{ij}$ is the Kronecker symbol. Thus, $h_F:=F-\o^p\psi_F$ and 
\begin{eqnarray}\label{representation}
\psi_F=:\sum_{j=0}^M \alpha_{F,j}\psi_j,
\end{eqnarray}
 where $\alpha_{F,j}:=\int_\O F\psi_j$ for any $0\leq j\leq M$. Moreover, notice that the coefficients are well-defined and \[|\alpha_{F,j}|\leq \|F\|_{L^q(\Omega,\o^{-1})}\|\psi_j\|_{L^p(\Omega,\o)},\]
for all $j$. In addition, using (\ref{representation}) we have
\begin{eqnarray}\label{weighted dense constant}
\|h_F\|_{L^q(\Omega,\o^{-1})}\leq \left(1+\sum_{j=0}^M \|\psi_j\|_{L^p(\O,\o)} \|\o^p\psi_j\|_{L^q(\O,\o^{-1})}\right)\|F\|_{L^q(\O,\o^{-1})}.
\end{eqnarray}

Now, in order to approximate the component $h_F$ by a function in $V_m$ we are going to need another orthonormal basis. Specifically,  let us take a cube $Q\subset\Omega$ that intersects a finite number of subdomains in $\{\Omega_t\}_{t\in\Gamma}$ and an orthonormal basis $\{\tilde{\psi}_i\}_{0\leq i\leq n}$ of $\P_m$ with respect to this other inner product 
\[\langle \psi,\varphi\rangle_Q=\int_Q \psi(x)\varphi(x) \o^p(x)\,\d x.\] 
Notice that in this case we use $Q$ instead of $\Omega$, however, $\tilde{\psi}_j$ is a polynomial in $\P_m$ and $\int_\Omega h_F\tilde{\psi}_j$ is still equal to zero for all $j$. Now, given $\epsilon>0$, and using that $\Gamma$ is countable and each $\O_t$ intersects a finite number of $\O_s$, let $\Omega_\epsilon\subset \Omega$ be
an open set that contains $Q$, also intersects a finite number of subdomains in $\{\Omega_t\}_{t\in\Gamma}$'s and 
\[\|(1-\chi_{\Omega_\epsilon})h_F\|_{L^q(\O,\o^{-1})}<\epsilon.\]

Thus, we define the function $G=g+\o^p\psi$, with $\psi:=\psi_F$ and 
\begin{eqnarray*}
g(x):=\chi_{\Omega_\epsilon}(x)h_F(x)+\sum_{i=0}^M\chi_Q(x)\o^p(x)\tilde{\psi}_i(x)\int_{\Omega\setminus\Omega_\epsilon}h_F(y)\tilde{\psi}_i(y)\,\d y.
\end{eqnarray*}
Observe that the support of $g$ intersects a finite number of $\O_t$'s, and $\int_\Omega g\tilde{\psi}_j=0$ for all $j$, thus $g\in V_m$. Moreover,
\begin{eqnarray*}
\|F-G\|_{L^q(\O,\o^{-1})}&=&\|h_F-g\|_{L^q(\Omega,\o^{-1})}\\
&\leq& \epsilon+\sum_{i=0}^M\|\chi_Q(x)\o^p(x)\tilde{\psi}_i(x)\int_{\Omega\setminus\Omega_\epsilon}h_F(y)\tilde{\psi}_i(y)\,\d y\|_{L^q(\Omega,\o^{-1})}\\
&\leq& \epsilon+\sum_{i=0}^M\int_{\Omega\setminus\Omega_\epsilon}|h_F(y)\tilde{\psi}_i(y)|\,\d y\|\tilde{\psi}_i\o^p\|_{L^q(Q,\o^{-1})}\\
&\leq& \epsilon\left(1+\sum_{i=0}^M\|\tilde{\psi}_i\|_{L^p(\Omega,\o)}    \|\tilde{\psi}_i\o^p\|_{L^q(Q,\o^{-1})}  \right).
\end{eqnarray*}

Finally, we only have to estimate the norm of $g$ by the norm of $G=g+\o^p\psi$. This representation is unique so we can assume that $g=h_G$ and $\psi=\psi_G$ defined above. Thus, from (\ref{weighted dense constant}) we have 
\begin{eqnarray*}
\|g\|_{L^q(\Omega,\o^{-1})}\leq \left(1+\sum_{j=0}^M \|\psi_j\|_{L^p(\O,\o)} \|\o^p\psi_j\|_{L^q(\O,\o^{-1})}\right)\|g+\o^p\psi\|_{L^q(\O,\o^{-1})}.
\end{eqnarray*}
In the particular case when $m=0$, the space $\P_0$ has dimension equal to 1 and we take the basis given by the vector 
\[\psi_0(x):=\frac{1}{(\o^p(\O))^{1/2}}\chi_\O(x),\]
where $\o^p(\O):=\int_\O \o^p$. Thus, $\|\psi_0\|_{L^p(\O,\o)} \|\o^p\psi_0\|_{L^q(\O,\o^{-1})}=1$ obtaining a constant that equals $2$. 
\end{proof}

The following is the main result of the section.

\begin{theorem}\label{Korn} If weighted Korn inequality (\ref{inf weighted local Korn}) is valid on $\O_t$, with an uniform constant $C_1$ for all $t\in\Gamma$, and there exists a $\P_0$-orthogonal  decomposition of any function $g$ in $V_0(\O,\o^{-1})$ subordinate to $\{\Omega_t\}_{t\in\Gamma}$, with the estimate
\begin{eqnarray*}
\sum_{t\in \Gamma} \|g_t\|^q_{L^q(\O_t,\o^{-1})}\leq C^q_0\|g\|^q_{L^q(\O,\o^{-1})},
\end{eqnarray*}
then, weighted Korn inequality (\ref{weighted local Korn}) is valid in $\Omega$. Namely, there exist a constant $C$ such that  
\begin{eqnarray}\label{optimal Korn}
\|D\uu\|_{L^p(\O,\o)}\leq C \|\varepsilon(\uu)\|_{L^p(\O,\o)}
\end{eqnarray}
is valid for any arbitrary vector field $\uu\in W^{1,p}(\O,\o)^n$, with $\int_{\Omega}\eta_{ij}(\uu)\,\o^p=0$ for $1\leq i<j\leq n$.
\end{theorem}

\begin{proof} The differential matrix of $\uu$ can be written as the sum of its symmetric part $\varepsilon(\uu)$ and its skew-symmetric part $\eta(\uu)$. Thus, in order to prove the theorem, it is necessary and sufficient to show that $\|\eta_{ij}(\uu)\|_{L^p(\O,\o)}\leq C \|\varepsilon(\uu)\|_{L^p(\O,\o)}$, for $1\leq i<j\leq n$. 

Now, given $t\in \Gamma$, we have 
\begin{eqnarray}\label{inf2 weighted local Korn}
\inf_{\alpha\in\P_0}\|\eta_{ij}(\uu)-\alpha\|_{L^p(\O_t,\o)}\leq C_1 \|\varepsilon(\uu)\|_{L^p(\O_t,\o)},
\end{eqnarray}
for any $1\leq i<j\leq n$, where $C_1$ is independent of $t$.

Let $g+\o^p\psi$ be an arbitrary function in $S_0$, with $\|g+\o^p\psi\|_{L^q(\O,\o^{-1})}\leq 1$. The function $\psi$ is simply a constant.  Thus, using that $\int_\O\eta_{ij}(\uu)\,\o^p=0$ and the existence of the $\P_0$-orthogonal decomposition we have
\begin{eqnarray*}
\int_{\Omega} \eta_{ij}(\uu)(g+\o^p\psi)&=&\int_{\Omega} \eta_{ij}(\uu)g=\int_{\Omega} \eta_{ij}(\uu)\sum_{t\in\Gamma} g_t\\
&=&\sum_{t\in\Gamma} \int_{\Omega_t} \eta_{ij}(\uu)g_t=\sum_{t\in\Gamma} \int_{\Omega_t} (\eta_{ij}(\uu)-\alpha)g_t=(I),
\end{eqnarray*}
for any arbitrary $\alpha\in\P_0$. Observe that the sum in the previous lines is finite as $g$ is a function in $V_0$.
Next, using H\"older inequality in $(I)$, the inequality (\ref{inf2 weighted local Korn}) on each $\Omega_t$ and, finally,  H\"older inequality for the sum, we obtain 
\begin{eqnarray*}
(I)&\leq&\sum_{t\in\Gamma} \inf_{\alpha\in\P_0} \|(\eta_{ij}(\uu)-\alpha)\|_{L^p(\O_t,\o)}\|g_t\|_{L^q(\O_t,\o^{-1})}\\
&\leq&\sum_{t\in\Gamma} C_1\|\varepsilon(\uu)\|_{L^p(\O_t,\o)}\|g_t\|_{L^{q}(\O_t,\o^{-1})}\\ 
&\leq&C_1 \left(\sum_{t\in\Gamma} \int_{\O_t} |\varepsilon(\uu)|^p\o^p\right)^{1/p}\left(\sum_{t\in\Gamma} \|g_t\|_{L^{q}(\O_t,\o^{-1})}^q\right)^{1/{q}}=(II).
\end{eqnarray*}
Now, we use the estimate of the decomposition given in the statement of the theorem, the estimate of the overlap of $\{\Omega_t\}_t$ and the estimate of the constant in Lemma \ref{weighted dense} 
\begin{eqnarray*}
(II)&\leq&C_1N^{1/p}\,C_0\,\|\varepsilon(\uu)\|_{L^p(\O,\o)} \|g\|_{L^q(\O,\o^{-1})}\\ 
&\leq&2C_1 N^{1/p}\,C_0\,\|\varepsilon(\uu)\|_{L^p(\O,\o)}.
\end{eqnarray*}

Finally, as $S_0$ is dense in $L^q(\Omega,\o^{-1})$, taking the supremum over all the functions $g+\o^p\psi$ in $S_0$ with $\|g+\o^p\psi\|_{L^q(\O,\o^{-1})}\leq 1$ we conclude 
\begin{eqnarray*}
\|\eta_{ij}(\uu)\|_{L^p(\O,\o)}&=&\sup_{g+\o^p\psi}\int_{\O} \eta_{ij}(\uu)(g+\o^p\psi)\leq 2N^{1/p}\,C_0\,C_1\,\|\varepsilon(\uu)\|_{L^p(\O,\o)}.
\end{eqnarray*}
Thus, 
\begin{eqnarray*}
\|D\uu\|_{L^p(\O,\o)}&\leq& \|\varepsilon(\uu)\|_{L^p(\O,\o)}+\|\eta(\uu)\|_{L^p(\O,\o)}\\
&\leq&\left(1+2n^{2/p}N^{1/p}\,C_0\,C_1\right) \|\varepsilon(\uu)\|_{L^p(\O,\o)}
\end{eqnarray*}
proving that weighted Korn inequality is valid on $\Omega$.
\end{proof}

\begin{remark}\label{estimate}
Notice that the proof of Theorem \ref{Korn} also gives an explicit constant for weighted Korn inequality (\ref{optimal Korn}) on $\Omega$. Indeed, we can take  
\[C= 1+2n^{2/p}N^{1/p}\,C_0\,C_1,\]
where $C_1$ is a uniform constant for weighed Korn inequality on each subdomain $\Omega_t$, $C_0$ is the constant involved in the estimation of the $\P_0$-decomposition and $N$ controls the overlap. 
\end{remark}

\section{A \texorpdfstring{$\P_0$}{P0}-decomposition on general domains}
\label{Decomposition on general domains}
\setcounter{equation}{0}

In this section we show the existence of a $\P_0$-decompositions subordinate to a covering $\{\O_t\}_{t\in\Gamma}$ of $\O$ if we have certain order on $\Gamma$. The construction of the $\P_0$-decomposition follows the ideas appearing in \cite{L}, where this kind of techniques involving decomposition of functions was used to prove the solvability in weighted Sobolev spaces of the equation ${\rm div}\,\uu=f$ on some irregular domains.

Let us denote by $G=(V,E)$ a graph with vertices $V$ and edges $E$. Graphs in these notes do not have neither multiple edges nor loops and the number of vertices in $V$ is at most countable. A {\it rooted tree} (or simply a tree) is a connected graph $G=(\Gamma,V)$ in which any two vertices are connected by exactly one simple path, and a {\it root} is simply a distinguished vertex $a\in\Gamma$. The set of vertices of a tree will be usually denoted by $\Gamma$ and we may say that $\Gamma$ has a rooted tree structure without specifying the set of edges $E$.  Moreover, if $G=(\Gamma,E)$ is a rooted tree, it is possible to define a {\it partial order} ``$\preceq$" in $\Gamma$ as follows: $s\preceq t$ if and only if the unique path connecting $t$ with the root $a$ passes through $s$. The {\it height} or {\it level} of any $t\in \Gamma$ is the number of vertices in $\{s\in\Gamma\,:\,s\preceq t\text{ with }s\neq t\}$. {\it The parent} of a vertex $t\in \Gamma$ is the vertex $s$ satisfying that $s\preceq t$ and its height is one unit smaller than the height of $t$. We denote the parent of $t$ by $t_p$. 
It can be seen that each $t\in\Gamma$ different from the root has a unique parent, but several elements on $\Gamma$ could have the same parent. Note that two vertices are connected by an edge ({\it adjacent vertices}) if one is the the parent of the other.

\begin{defi} \label{Decomp of Omega}
Let $\Omega\subset\R^n$ be a bounded domain and $\{\Omega_t\}_{t\in\Gamma}$ a covering of $\Omega$. We say that $\{\Omega_t\}_{t\in\Gamma}$ is a {\it tree covering} of $\Omega$ if $\Gamma$ is the set of vertices of a rooted tree, with root $a\in\Gamma$, such that for any $t\in \Gamma$, with $t\neq a$, there exists an open cube $B_t\subseteq\O_{t}\cap\O_{t_p}$ where the collection $\{B_t\}_{t\neq a}$ is pairwise disjoint. 
\end{defi}

The tree structure on $\Gamma$ gives a certain notion of geometry to $\O$. We are interested in graph structures which are consistent with the geometry that we already have in $\O$. In Section \ref{John}, we will show the existence of an appropriate tree structure on $\Gamma$, where $\{\O_t\}_{t\in\Gamma}$ is a dilation of a Whitney decomposition of a John domain $\O$. Similar constructions have been developed in \cite{L} for H\"older-$\alpha$ domains and other examples.
 
\begin{defi}\label{T and W} 
Given a tree covering $\{\O_t\}_{t\in\Gamma}$ of $\O$ we define {\it the Hardy type operator} $T$ as follows:
\begin{eqnarray}\label{Ttree}
Tg(x):=\sum_{a\neq t\in\Gamma}\dfrac{\chi_{t}(x)}{|W_t|}\int_{W_t}|g|,
\end{eqnarray} 
where $\displaystyle{W_t=\bigcup_{s\succeq t} \O_s}$ and $\chi_t$ is the characteristic function of $B_t$ for all $t\neq a$. 
\end{defi}
We may refer to $W_t$ by {\it the shadow of }$\Omega_t$.

The next lemma is a fundamental result that proves the continuity of the operator $T$. This result was shown in \cite{L} (Lemma 3.1).

\begin{lemma}\label{Ttreecont} The operator $T:L^q(\O)\to L^q(\O)$ defined in (\ref{Ttree}) is continuous for any $1<q< \infty$. Moreover, its norm is bounded by 
\[\|T\|_{L^q\to L^q} \leq 2\left(\dfrac{qN}{q-1}\right)^{1/q}.\]
\end{lemma}

It is well-known that the Hardy-Littlewood maximal operator plays an important role in the theory of singular integral operators in weighted spaces. This Hardy type operator plays a similar role when we want to define decompositions of functions in weighted spaces. Another article where Hardy operators have been used to prove weighted Korn inequality is \cite{AO}, where the authors deal with domains which have an external cusp.

\begin{theorem}\label{Decomp} Let $\Omega\subset\R^n$ be a bounded domain with a tree covering $\{\O_t\}_{t\in \Gamma}$. Given $g\in L^1(\Omega)$ such that $\int_\O g=0$ and $supp(g)\cap \O_s\neq \emptyset$ for a finite number of $s\in\Gamma$, there exists $\{g_t\}_{t\in\Gamma}$, a  $\P_0$-decompositions of $g$ subordinate to $\{\Omega_t\}_{t\in\Gamma}$ (see Definition \ref{Definition decomposition}).

Moreover, let us take $t\in\Gamma$. If $x\in B_s$ where $s=t$ or $s_p=t$ then 
\begin{eqnarray}\label{P02}
|g_t(x)|\leq |g(x)|+\tfrac{|W_s|}{|B_s|}Tg(x),
\end{eqnarray}
where $W_t$ denotes the shadow of $\Omega_t$ previously defined. Otherwise
\begin{eqnarray}\label{P01}
|g_t(x)|\leq |g(x)|.
\end{eqnarray}
\end{theorem}

\begin{proof}
Let $\{\phi_t\}_{t\in\Gamma}$ be a partition of the unity subordinate to $\{\O_t\}_{t\in\Gamma}$. Namely, a collection of smooth functions such that $\sum_{t\in\Gamma} \phi_t=1$, $0\leq \phi_t\leq1$ and $supp(\phi_t)\subset\Omega_t$. Thus, $g$ can be cut-off into $g=\sum_{t\in\Gamma} f_t$ by taking $f_t=g\phi_t$. This decomposition verifies (1) and (2) in Definition \ref{Definition decomposition} but (3) may not be satisfied. Thus, we will make some modifications to obtain the orthogonality with respect to $\P_0$.

\bs

The new collection of cutting functions that preserves the orthogonality of $g$ with respect to $\P_0$ is $\{g_t\}_{t\in\Gamma}$, which is defined by
\begin{eqnarray}\label{P0-decomposition}
g_t(x):=f_t(x)+\left(\sum_{s:\,s_p=t}h_s(x)\right)-h_t(x), 
\end{eqnarray}
where 
\begin{eqnarray}\label{h0}
h_s(x):=\dfrac{\chi_{s}(x)}{|B_{s}|}\int_{W_{s}}\sum_{k\succeq s}f_k. 
\end{eqnarray}

We denote by $\chi_t$ the characteristic function of $B_t$.
The sum in (\ref{P0-decomposition}) is indexed over every $s\in\Gamma$ such that $t$ is the parent of $s$. In the particular case when $t$ is the root of $\Gamma$,  (\ref{P0-decomposition}) means 
\[g_a(x)=g_a(x)+\sum_{s:\,s_p=a}h_s(x).\]

Note that the functions $h_s$ in (\ref{h0}) are well-defined because of the integrability of $g$. Moreover, $h_s\not\equiv 0$ only if $f_t\not\equiv 0$ for some $a\preceq s\preceq t$. Thus, $h_s\not\equiv 0$ for a finite number of $s\in\Gamma$. In addition, we have the following immediate properties
\[supp(h_s)\subset B_s\]
\begin{eqnarray}\label{TP0}
|h_s(x)|\leq \tfrac{|W_s|}{|B_s|}\chi_s(x) Tg(x)\text{ for all }x\in\Omega.
\end{eqnarray}
Next, using (\ref{TP0}) we can conclude that $|g_t(x)|\leq |g(x)|+ \tfrac{|W_s|}{|B_s|}Tg(x)$, for any $x\in B_s$ with $s=t$ or $s_p=t$ and $|g_t(x)|\leq |g(x)|$ otherwise, proving (\ref{P02}) and (\ref{P01}).

Let us continue by showing that $g(x)=\sum_{t\in\Gamma} g_t(x)$ for all $x$. Take $x\in \Omega\setminus \bigcup_{k\in\Gamma} B_k$, then $g_t(x)=f_t(x)$, for all $t\in\Gamma$, and \[\sum_{t\in\Gamma} g_t(x)=\sum_{t\in\Gamma} f_t(x)=g(x).\]
Otherwise, if $x$ belongs to $B_{\tilde{k}}$ for $\tilde{k}\in\Gamma$, it can be observed that $g_t(x)=f_t(x)$ for all $t$ such that  $t\neq \tilde{k}$ and $t\neq \tilde{k}_p$. We are using that the cubes $B_s$ are pairwise disjoint. Moreover,  
\begin{eqnarray*}
g_{\tilde{k}}(x)&=&f_{\tilde{k}}(x)-h_{\tilde{k}}(x)\\
g_{\tilde{k}_p}(x)&=&f_{\tilde{k}_p}(x)+h_{\tilde{k}}(x).
\end{eqnarray*}
Then, $\sum_{t\in\Gamma} g_t(x)=g(x)$ for all $x$.

The second property in definition \ref{Definition decomposition} follows by observing that the parent of each $s$ in (\ref{P0-decomposition}) is $t$, then $B_s\subseteq \Omega_s\cap\Omega_t$. Thus,  $supp(g_t)\subseteq\Omega_t$.

Finally, in order to prove that $g_t$ is orthogonal to $\P_0$ for all $t\in\Gamma$ observe that $k\succeq t$ if and only if $k\succeq s$, with $s_p=t$, or $k=t$. Thus, 
\begin{eqnarray*}
\int h_t&=&\int_{W_s}\sum_{k\succeq t}f_k=\int_{\Omega_t}f_t\ +\sum_{s:\,s_p=t}\int_{W_s}\sum_{k\succeq s}f_k\\
&=&\int_{\Omega_t}f_t\ +\sum_{s:\,s_p=t}\int h_s.
\end{eqnarray*}
Then, $\int g_t=0$ for all $t\neq a$. Finally, $\int g_a=\int g=0$. 

\end{proof}

\section{Korn inequality and more on John domains}
\setcounter{equation}{0}
\label{John}

Let $\Omega\subset\R^n$ be a bounded John domain. In the first part of the section, and in order to use the results stated in Section \ref{Inequalities} and Section \ref{Decomposition on general domains}, we will show that there exists a tree covering $\{\O_t\}_{t\in\Gamma}$ of $\O$ for which it is possible to estimate the ratio $\tfrac{|W_t|}{|B_t|}$ for any $t\in\Gamma\setminus\{a\}$. This covering also satisfies (\ref{overlapping}) and that each $\O_t$ intersects a finite number of $\O_s$, with $s$ in $\Gamma$.

A Whitney decomposition of $\O$ is a collection $\{Q_t\}_{t\in\Gamma}$ of closed dyadic cubes whose interiors are pairwise disjoint, which verifies
\begin{enumerate}
\item $\O=\bigcup_{t\in\Gamma}Q_t$,
\item $\text{diam}(Q_t) \leq \rho(Q_t,\partial\Omega) \leq 4\text{diam}(Q_t)$,
\item $\frac{1}{4}\text{diam}(Q_s)\leq \text{diam}(Q_t)\leq 4\text{diam}(Q_s)$, if $Q_s\cap Q_t\neq \emptyset$.
\end{enumerate}
Two different cubes $Q_s$ and $Q_t$ with $Q_s\cap Q_t\neq \emptyset$ are called {\it neighbors}. Notice that two neighbors may have an intersection with dimension less than $n-1$. For instance, they could be intersecting each other in a one-point set. We say that $Q_s$ and $Q_t$ are \mbox{$(n-1)$}-neighbors if $Q_s\cap Q_t$ is a $n-1$ dimensional face.
This kind of covering exists for any proper open set in $\R^n$ (see \cite{S} for details). Moreover, each cube $Q_t$ has no more than $12^n$ neighbors. And, if we fix $0<\epsilon<\frac{1}{4}$ and define $Q_t^*$ as the cube with the same center as $Q_t$ and side length $(1+\epsilon)$ times the side length of $Q_t$, then, $Q_t^*$ touches $Q^*_s$ if and only if  $Q_t$ and $Q_s$ are neighbors. Thus, each expanded cube has no more than $12^n$ neighbors and $\sum_{t\in\Gamma}\chi_{Q_t^*}(x)\leq 12^n$.

\begin{defi}\label{Boman chain condition} A bounded domain $\O\subset\R^n$ is said to satisfy the {\it Boman chain condition} if there exists a Whitney decomposition $\{Q_t\}_{t\in\Gamma}$ of $\O$, with a distinguished cube $Q_a$, and $\lambda>1$ such that for any cube $Q_t$, with $t\in\Gamma$, there is a chain of cubes pairwise different $Q_{t,0},Q_{t,1},\cdots,Q_{t,\kappa}$ such that $Q_{t,0}=Q_t$, $Q_{t,\kappa}=Q_a$ and 
\begin{eqnarray}\label{Boman} 
Q_{t,i}\subseteq \lambda Q_{t,j},
\end{eqnarray}
for all $0\leq i\leq j\leq \kappa$, where $\kappa=\kappa(t)$.

Moreover, two consecutive cubes $Q_{t,i-1}$ and $Q_{t,i}$ in this chain are \mbox{$(n-1)$}-neighbors.
\end{defi}

This kind of conditions were first introduced by Boman in \cite{B}. Later, Buckley et al. proved in \cite{BKL}, in a very general context, that the condition introduced by Boman characterizes John domains. The formulation in Definition \ref{Boman chain condition} is slightly different from the one in \cite{BKL}, as we use that (\ref{Boman}) mis valid for all $0\leq i \leq j$, and not just $i=0$ as in \cite{BKL}. Thus, to prove that any bounded Jonh domain verifies this definition we use Theorem 3.8 in \cite{DRS}. 

\begin{lemma}\label{Existence of a chain} Any bounded John domain $\O\subset\R^n$ satisfies the Boman chain condition in Definition \ref{Boman chain condition}.
\end{lemma}
\begin{proof} Given a Whitney decomposition $\{Q_t\}_{t\in\Gamma}$ of $\O$ and following \cite{DRS}, there is a distinguished cube $Q_a$, and $\lambda>1$ such that for each cube $Q_t$ there is a chain of cubes pairwise different $Q_{t,0},Q_{t,1},\cdots,Q_{t,\kappa}$ that connects $Q_t$ with $Q_a$ and satisfies (\ref{Boman}). Let us modify this chain in order to have the property that two consecutive cubes in the chain are \mbox{$(n-1)$}-neighbors. Thus, suppose that $F:= Q_{t,i-1}\cap Q_{t,i}$  has dimension $d$ in $0\leq d\leq n-2$. Then, we take $n-d-1$ Whitney cubes intersecting $F$ such that two consecutive cubes in the chain $Q_{t,i-1},Q_1, \cdots, Q_{n-d-1},Q_{t,i}$  are \mbox{$(n-1)$}-neighbors. Moreover, from the third condition in the Whitney decomposition we have that the dilation by a constant $C_n$ of each cube in this small list contains the other ones. Thus, repeating this process between two consecutive cubes in $Q_{t,0},Q_{t,1},\cdots,Q_{t,\kappa_t}$ and replacing $\lambda$ by $C_n\lambda$ in (\ref{Boman}), we obtain a Boman chain of Whitney cubes where two consecutive cubes are \mbox{$(n-1)$}-neighbors. The pairwise different condition is easily recovered, in case it is necessary, by removing the cubes in the chain between the repeated cubes.
\end{proof}
\begin{remark} It is well known that if $\O$ satisfies the Boman chain condition with a distinguished cube $Q_a$, then we can take as a distinguished cube any arbitrary cube in the Whitney decomposition. However, the constant $\lambda$ in (\ref{Boman}) may vary.
\end{remark}

In order to define an appropriate tree covering of $\O$, we have to prove that John domains satisfy the new condition stated below which is apparently richer than the Boman chain condition.

\begin{defi}\label{Boman tree condition} Let $\O\subset\R^n$ be a bounded domain. We say that $\O$ satisfies the {\it Boman tree condition} if there exists a Whitney decomposition $\{Q_t\}_{t\in\Gamma}$, where $\Gamma$ has a rooted tree structure, that satisfies  
\begin{eqnarray}\label{Boman tree} 
Q_s\subseteq K Q_{t},
\end{eqnarray}
for any $s,t\in\Gamma$, with $s\succeq t$.
Moreover, if two vertices $t$ and $s$ are adjacent in $\Gamma$ then $Q_t$ and $Q_s$ must be \mbox{$(n-1)$}-neighbors.
\end{defi}

\begin{lemma}\label{Existence of a tree} Boman chain condition and Boman tree condition are equivalent.
\end{lemma}

The reverse of the equivalence in Lemma \ref{Existence of a tree} is obtained by taking $Q_a$ as the distinguished cube, where $a$ is the root of $\Gamma$.  Thus, given $Q_\kappa$, with $\kappa\in\Gamma$, we have $Q_s\subseteq K Q_t,$ for all $a\preceq t\preceq s\preceq \kappa$. Observe that in this case the chain starts at $Q_a$, instead of $Q_{\kappa}$ as in (\ref{Boman}), and ends at $Q_\kappa$.
The other implication is shown in the Appendix and follows some ideas by A. A. Vasil'eva (see \cite{V}). 

Let $\O\subset\R^n$ be a bounded John domain. Then, from Lemma \ref{Existence of a chain}  and  Lemma \ref{Existence of a tree}, we know that there exists a Whitney decomposition $\{Q_t\}_{t\in\Gamma}$ fulfilling all the properties in Definition \ref{Boman tree condition}. Thus, we define a covering $\{\O_t\}_{t\in\Gamma}$ of $\O$ by 
\begin{eqnarray}\label{Covering}
\O_t:=\frac{17}{16}Q_t^\circ,
\end{eqnarray} 
where $\frac{17}{16}Q_t^\circ$ denotes the open cube with the same center as $Q_t$ and side length $\frac{17}{16}$ times the side length of $Q_t$. 

\begin{corollary}\label{tree covering for John} The covering $\{\O_t\}_{t\in\Gamma}$ of the bounded John domain $\O$ defined in (\ref{Covering}) is a tree covering with 
\begin{eqnarray}\label{Bt}
\text{diam}(\O_t)\leq C_n\,\text{diam}(B_t)
\end{eqnarray}
and 
\begin{eqnarray}\label{ratio}
\text{diam}\left(\bigcup_{s\succeq t}\O_s\right)\leq K\,\text{diam}(\O_t),
\end{eqnarray}
for any $t\in\Gamma$ ($t\neq a$ in the first inequality), where $K$ is the constant in (\ref{Boman tree}). 

Moreover, overlapping condition (\ref{overlapping}) is satisfied with $N=12^n$, each $\Omega_t$ intersects a finite number of $\Omega_s$ with $s\in\Gamma$, and 
\begin{eqnarray}\label{distance comparison}
\frac{1}{C_n}\,\text{diam}(\O_t) \leq \rho(\O_t,\partial\Omega) \leq C_n\,\text{diam}(\O_t).
\end{eqnarray}
\end{corollary}
\begin{proof} Regarding (\ref{ratio}), let us observe that (\ref{Boman tree}) is also valid for the cubes in $\{\O_t\}_{t\in\Gamma}$ as we are dilating the cubes in $\{Q_t\}_{t\in\Gamma}$ by the same factor. Thus,
\begin{eqnarray*}
\O_s\subseteq K \O_t,
\end{eqnarray*}
for any $t,s\in\Gamma$, with $t\preceq s$, obtaining (\ref{ratio}). The rest is a straightforward calculation except the existence of the pairwise disjoint collection $\{B_t\}_{t\neq a}$ satisfying (\ref{Bt}). We know that $Q_t$ and $Q_{t_p}$ are \mbox{$(n-1)$}-neighbors. Thus, $F_t:=Q_t\cap Q_{t_p}$ is a $n-1$ dimensional face of the smallest of this two cubes. We name by $\alpha_t$ the centroid of $F_t$. Let us use the distance $d_1(x,y):=\max_{1\leq i\leq n}|x_i-y_i|$, which is more convenient than $d(x,y)=\sqrt{\sum_{i}(x_i-y_i)^2}$ in this context. Moreover, let us use the side length of $Q_t$, which is denoted by $l(Q_t)$, instead of diam$(Q_t)$. Thus, using the third condition in the Whitney decomposition,  it can be seen that 
\begin{eqnarray*}
d_1(\alpha_t,\alpha_s)\geq \frac{1}{8}l(Q_t),
\end{eqnarray*}
for all $s\in\Gamma\setminus\{a,t\}$. Thus, if we define $B_t$ as the open cube with center at $\alpha_t$ and side length equal to $\frac{l(Q_t)}{8}$, we obtain a collection of pairwise disjoint cubes. However, it is also required that $B_t\subset \O_t\cap \O_{t_p}$. Thus, we take $B_t$ with length side equal to $\frac{l(Q_t)}{64}$ which satisfies the required conditions. Then,  (\ref{Bt}) holds with $C_n=64$. 
\end{proof}

The next lemma will be used to prove the weighted estimate for the $\P_0$-orthogonal decomposition that appears in Theorem \ref{Korn}.

\begin{lemma}\label{weighted T}
Let $\O\subset\R^n$ be a bounded John domain, $\{\O_t\}_{t\in\Gamma}$ the tree covering defined in (\ref{Covering}), and $\beta\geq 0$. Then, the operator $T$ defined in (\ref{Ttree}) and subordinate to $\{\O_t\}_{t\in\Gamma}$ is continuous from $L^q(\O,\rho^{-\beta})$ to itself, where $\rho$ is the distance to the boundary of $\O$. Moreover, its norm is bounded by 
\[\|T\|_{L\to L} \leq C_n^\beta \left(\frac{qN}{q-1}\right)^{1/q}K^\beta,\]
where $L$ denotes $L^q(\O,\rho^{-\beta})$. The constant $K$ is the one in (\ref{Boman tree}) and $N=12^n.$
\end{lemma}

It can be seen, after multiplying by an appropriate constant, that the Hardy-Littlewood maximal operator pointwise bounds the Hardy type operator $T$ defined by using the tree covering introduced in (\ref{Covering}). Thus $T$ is continuous from $L^p(\O,\o)$ to itself if $\o$ belongs to the $A_p$ Muckenhoupt class. However, arbitrary positive powers of the distance to $\partial\O$ are not necessary in this class. Thus, we have to prove the weighted continuity of $T$ in a different way.

\begin{proof} Given $g\in L^q(\O,\rho^{-\beta})$ we have
\begin{eqnarray*} 
\int_\O|Tg(x)|^q \rho^{-q\beta}(x)\,\d x&=&\int_{\O} \rho^{-q\beta}(x) \left|\sum_{a\neq t\in\Gamma}\frac{\chi_t(x)}{|W_t|}\int_{W_t}|g(y)|\,\d y\right|^q\,\d x\\
&=& \int_{\O} \rho^{-q\beta}(x) \left|\sum_{a\neq t\in\Gamma}\frac{\chi_t(x)}{|W_t|}\int_{W_t}|g(y)|\rho^{-\beta}(y)\,\rho^{\beta}(y)\,\d y\right|^q\,\d x=(1)
\end{eqnarray*}

Now, given $y\in W_t$ there exists $s\succeq t$ such that $y\in\O_s$. Thus, it can be seen that 
\[\rho(y)\leq C_n {\rm diam}(\O_s)\leq C_n\,K\,{\rm diam}(\O_t).\]
Then, using that $\beta$ is nonnegative we have
\[\rho^\beta(y)\leq C_n^\beta\,K^\beta\,({\rm diam}(\O_t))^\beta \leq C_n^\beta\,K^\beta\,\rho^\beta(x),\]
for all $x\in B_t$. Recall that $\chi_t$ is the characteristic function of $B_t\subset \O_t$ and $\text{diam}(\O_t)$ is comparable to $\rho(\O_t,\partial\Omega)$. Thus, 
\begin{eqnarray*} 
(1)&=& C_n^{q\beta}\,K^{q\beta} \int_{\O} \rho^{-q\beta}(x) \left|\sum_{a\neq t\in\Gamma}\frac{\chi_t(x)\rho^\beta(x)}{|W_t|}\int_{W_t}|g(y)|\rho^{-\beta}(y)\,\d y\right|^q\,\d x\\
&=& C_n^{q\beta}\,K^{q\beta}\int_{\O} \left|\sum_{a\neq t\in\Gamma}\frac{\chi_t(x)}{|W_t|}\int_{W_t}|g(y)|\rho^{-\beta}(y)\,\d y\right|^q\,\d x\\
&=&C_n^{q\beta}\,K^{q\beta} \int_{\O} \left|T(g\rho^{-\beta})\right|^q\,\d x=(2)
\end{eqnarray*}

Finally, $g\rho^{-\beta}$ belongs to $L^q(\O)$ and $T$ is continuous from $L^q(\O)$ to itself (see Lemma \ref{Ttreecont}), thus 
\begin{eqnarray*} 
(2)\leq \left(C_n^{q\beta}\, 2^q \frac{qN}{q-1}\right)K^{q\beta}\, \|g\|^q_{L^q(\O,\rho^{-\beta})}.
\end{eqnarray*}
\end{proof}

\begin{proof} [Proof of Theorem \ref{Korn in John}] Using Theorem \ref{Decomp} we can observe that there exists a $\P_0$-decomposition $\{g_t\}_{t\in\Gamma}$ of any integrable function $g$. This decomposition is subordinate to the tree covering $\{\O_t\}_{t\in\Gamma}$ defend in (\ref{Covering}). Moreover, it verifies (\ref{P02}) that, in this case, implies that 
\begin{eqnarray*}
|g_t(x)|\leq |g(x)|+C_n\, K^n\,Tg(x),
\end{eqnarray*}
for any $x\in\O_t$ with $t\in\Gamma$. Thus, by a straightforward calculation we have 
\begin{eqnarray*}
\int_{\O_t}|g_t(x)|^q\rho^{-q\beta}(x)\,\d x\leq 2^{q-1}\left(\int_{\O_t}|g(x)|^q\rho^{-q\beta}(x)\,\d x+C_n^q\, K^{qn}\int_{\O_t}|Tg(x)|^q\rho^{-q\beta}(x)\,\d x\right).
\end{eqnarray*}
Next, by using the bound on the overlap and Lemma \ref{weighted T}, we have the estimate required in Theorem \ref{Korn}
\begin{eqnarray*}
\sum_{t\in\Gamma} \|g_t\|^q_{L^q(\O_t,\rho^{-\beta})}&\leq& 2^{q-1}N\left( \|g\|^q_{L^q(\O,\rho^{-\beta})}+c^q_n\, K^{qn} \|Tg\|^q_{L^q(\O,\rho^{-\beta})}\right)\\
&\leq& 2^{q-1}N\left(1+c^q_n\, K^{qn}\left(C_n^{q\beta}\, 2^q \frac{qN}{q-1}\right)K^{q\beta}  \right)\|g\|^q_{L^q(\O,\rho^{-\beta})}.
\end{eqnarray*}
Moreover, being consistent with the notation used in Theorem \ref{Korn} we have that \[C_0=C_{n,p,\beta}\,K^{n+\beta}.\]

Finally, we show the validity of Korn inequality (\ref{inf weighted local Korn}) on $\O_t$, with $\o=\rho^{\beta}$, with a constant $C_{p,n}$ independent of $t\in\Gamma$. Using that the distance from $\O_t$ to the boundary of $\O$ is comparable to $\text{diam}(\O_t)$, it is easy to show that the weight is comparable to a constant over $\O_t$, indeed, 
\[\frac{1}{C_n}\,\text{diam}(\O_t) \leq \rho(x) \le C_n\,\text{diam}(\O_t),\]
for all $x\in\O_t$. Moreover, Korn inequality (\ref{inf weighted local Korn}) with $\o=1$ is valid on any cube $\O_t$ with uniform constant. Thus, 
\begin{eqnarray*}
\inf_{\varepsilon(\ww)=0}\|D(\vv-\ww)\|_{L^p(\O_t,\rho^\beta)}&\leq& C_n^\beta\,\text{diam}(\O_t)^{\beta} \inf_{\varepsilon(\ww)=0}\|D(\vv-\ww)\|_{L^p(\O_t)}\\
&\leq&C_n^\beta\,\text{diam}(\O_t)^{\beta} C_{p,n} \|\varepsilon(\vv)\|_{L^p(\Omega_t)}\\
&\leq&C_n^\beta\, C_{p,n} \|\varepsilon(\vv)\|_{L^p(\Omega_t,\rho^\beta)},
\end{eqnarray*}
with a constant $C_1=C_{p,n,\beta}$. Thus, the validity of (\ref{weighted Korn John}) and the estimate of its constant follows from Theorem \ref{Korn} and Remark \ref{estimate}.
\end{proof}

\subsection{Weighted solutions of divergence problem on John domains}

In this subsection, we basically combine Theorem 3.2 from \cite{L} and Lemma \ref{Existence of a tree} to show the existence of a weighted solution of $\di\uu=f$ on John domains. This problem is basic for the theoretical and numerical analysis of the Stokes equations in $\O$ and has been widely studied (see \cite{G,ADM,Bog,Du,DMRT,L} and references therein). The solutions belong to $W^{1,q}_0(\O,\rho^{-\beta})^n$ which is defined as the 
closure of $C_0^\infty(\O)^n$ with the norm 
\begin{eqnarray*}
\|\uu\|_{W^{1,q}_0(\O,\rho^{-\beta})^n}:=\|D\uu\|_{L^q(\O,\rho^{-\beta})^{n\times n}}.
\end{eqnarray*}

\begin{theorem}
Let $\O\subset\R^n$ be a bounded John domain with $n\geq 2$, $1<q<\infty$ and $\beta\in\R_{\geq 0}$. Given $f\in L^q(\O,\rho^{-\beta})$, with $\int_\O f=0$, there exists a solution $\uu\in W^{1,q}_0(\O,\rho^{-\beta})^n$ of $\di\uu=f$ that satisfies
\begin{eqnarray*}
\|D\uu\|_{L^q(\O,\rho^{-\beta})}\leq C_{n,q,\beta} K^{n+\beta} \|f\|_{L^q(\O,\rho^{-\beta})},
\end{eqnarray*}
where $\rho(x)$ is the distance to the boundary of $\Omega$  and $K$ is the constant in (\ref{Boman tree}).
\end{theorem}
\begin{proof} Let us show that the hypothesis in Theorem 3.2, \cite{L}, are fulfilled in this case. First, notice that the notation that we use in this article for $p$ and $q$ is swapped in \cite{L}.  Let $\{\O_t\}_{t\in\Gamma}$ be a tree covering as in Lemma \ref{tree covering for John}. For being a tree covering it satisfies {\bf (b)}. $\{\O_t\}_{t\in\Gamma}$ is obtained by expanding a Whitney decomposition which implies {\bf (a)} and {\bf (c)}, with $N=12^n$. Condition {\bf (d)} involves a weight $\o$ which depends on the geometry of $\O$:
\begin{equation*}
\o(x):=\left\{
  \begin{array}{l l}
     \dfrac{|B_t|}{|W_t|} & \quad \text{if $x\in B_t$ for some $t\in\Gamma$, $t\neq a$}\\
     \\
     1 & \quad \text{otherwise}.\\
   \end{array} \right.
\end{equation*}
Now, from (\ref{ratio}) it follows that $\o(x)\geq \frac{1}{C_nK^n}$ for any $x\in\O$. Thus, by taking $\bar{\o}:=1$ and $M_1:=C_nK^n$
we have {\bf (d)}. In order to prove {\bf (e)} we define $\hat{\o}:=\rho^{-\beta}$, and use that $\rho$ is comparable to diam$(\O_t)$ over $\O_t$ ((1) in Corollary \ref{tree covering for John}). Thus, using that there are solutions for the divergence problem on cubes with uniform constant we have that given $t\in\Gamma$ and $g\in L^q(\O_t,\rho^{-\beta})$, with vanishing mean value, there exists a solution  $\vv\in W_0^{1,q}(\O_t,\rho^{-\beta})^n$ of $\di\vv=g$ with 
\begin{eqnarray*}
\|D\vv\|_{L^q(\O_t,\rho^{-\beta})}\leq C_{n,\beta} \|g\|_{L^q(\O_t,\rho^{-\beta})}.
\end{eqnarray*}
Thus, $M_2$ is a constant that depends only on $n$ and $\beta$.

Finally, {\bf (f)} follows from Lemma \ref{weighted T} with $M_T=C_{n,q,\beta}\,K^{\beta}$. 

The estimate of the constant follows from the estimate in Theorem 3.2 in \cite{L}.
\end{proof}

\begin{appendices}

\section{Boman chain implies Boman tree condition}
\setcounter{equation}{0}
\numberwithin{equation}{section}

This section is devoted to prove Lemma \ref{Existence of a tree}. 

According to the previous section, $\{Q_t\}_{t\in\Gamma}$ denotes a Whitney decomposition of a bounded domain $\O\subset\R^n$ that satisfies Boman condition (\ref{Boman}). The center cube $Q_a$ can be arbitrarily chosen. Thus, we take it with the biggest size. Moreover, without loss of generality and in order to simplify the notation we are going to assume that its side length is 1. For any $s\in\Gamma$, we denote by $l_s$ the side length of $Q_s$. In addition, the elements in the covering are dyadic cubes, thus $l_s=2^{-m_s}$, where $m_s$ is a nonnegative integer number. The integer number $m_s$ can also be denoted by $m(Q_s)$. For example, $m(Q_a)=m_a=0$.

Let $G=(V,E)$ be a connected graph. Given $v,v'\in V$ we define the distance $k(v,v')$ as the minimal $j\in\N_0$ such that there exists a simple path $(v_0,v_1,\cdots,v_j)$ with length $j$ that connects $v$ with $v'$. Namely, $v_0=v$, $v_j=v'$, and the vertices $v_i$ and $v_{i+1}$ are adjacent. The function $k$ depends on $V$ and $E$.

\begin{lemma}\label{directed graph} Let $G=(V,E)$ be a connected graph with a distinguished vertex $v_\ast\in V$. The graph also satisfies that $k(v,v_\ast)\leq k$ for all $v\in V$, where $k$ is a fixed value in $\N$. Then, there exists a subgraph $\tilde{G}=(V,\tilde{E})$ with the same vertices which is a rooted tree with root $v_\ast$ such that $\tilde{k}(v,v_\ast)\leq k$ for all $v\in V$, where $\tilde{k}$ is the distance for the new graph $\tilde{G}$.
\end{lemma}
\begin{proof} The rooted tree $\tilde{G}$ is obtained by eliminating edges from $E$ by using an inductive argument. Indeed, we are going to define a collection $G_i:=(V_i,E_i)$ of subgraphs of $G$ for each $0\leq i\leq k$. The set of vertices $V_i$ has the vertices $v\in G$ such that $k(v,v_\ast)\leq i$. If $k_i$ denotes the distance between vertices in $G_i$, we define $E_i$ inductively so that $G_i$ is a subtree of $G_{i+1}$ and $k(v,v_\ast)=k_i(v,v_\ast)$ for all $v\in V_i$. 

Thus, we define $V_0=\{v_\ast\}$ and $E_0=\emptyset$. Next, given $1\leq i\leq k$, the process consists on taking exactly one edge that joins each vertex in $V_i\setminus V_{i-1}$ with $V_{i-1}$ and eliminating the other edges. 
\end{proof}

\begin{lemma}\label{Length of the chain}
Let $\{Q_t\}_{t\in\Gamma}$ be a Whitney decomposition of $\O$ satisfying (\ref{Boman}). Then, there exists a tree structure in $\Gamma$ such that for all $t,t'\in\Gamma$, with $t'\succeq t$, it follows:
\begin{eqnarray}\label{QH distance}
k(t,t')\leq l_{\ast}(m_{t'}-m_t)+k_{\ast},
\end{eqnarray}
where 
\begin{eqnarray*}
l_{\ast}&:=&(1+\lambda^{2n})(2+\log_2(\lambda))+1\\
k_{\ast}&:=&(1+\lambda^{2n})(2+\log_2(\lambda))+l_\ast(1+\log_2(\lambda)).
\end{eqnarray*}
The constant $\lambda$ is the one introduced in (\ref{Boman}). In addition, if two vertices $s$ and $t$ are adjacent then $Q_s$ and $Q_t$ must be neighbors.\end{lemma}
\begin{proof} We will prove this result by using a inductive argument. As we mentioned before, we are assuming that $Q_a$ has maximal side length equal to 1. Thus, we will define a collection of rooted trees $G_m=(\Gamma_m,E_m)$ for any $m\in\N_{\geq -1}$ such that $G_m$ is a subgraph of $G_{m+1}$ (i.e. $\Gamma_m\subseteq \Gamma_{m+1}$ and $E_m\subseteq E_{m+1}$) and  all of them are subgraphs of $G_{\O}=(\Gamma,E_{\O})$, where two vertices $t,t'\in\Gamma$ are adjacent in $G_\O$ if and only if $Q_t$ and $Q_{t'}$ are neighbors. Moreover, $\bigcup_m\Gamma_m=\Gamma$. 

The inductive hypothesis that we will use says:
\begin{itemize}
\item[{\rm (h1)}] $\Gamma_m$ contains all the cubes $Q_t$ with $m_t=m$.
\item[{\rm (h2)}] $m_t\leq m+1+\log_2(\lambda)$ for all $t\in \Gamma_m$. 
\item[{\rm (h3)}] If $t,t'\in \Gamma_m\setminus \Gamma_{m-1}$ with $t'\succeq t$, then $k(t,t')\leq \lambda^{2n}$.
\item[{\rm (h4)}] Condition (\ref{QH distance}) is satisfied for any $t,t'\in\Gamma_m$.
\end{itemize}

Let us start by defining $G_{-1}$ which has $\Gamma_{-1}:=\{a\}$ and $E_{-1}:=\emptyset$. 
It can be easily checked that $G_{-1}$ is a subtree of $G_\O$ that satisfies (h1) to (h4). 
So, let us suppose that we have a collection $G_{-1},\, G_{0},\,\cdots,G_{m-1}$ for $m\geq 0$ verifying all the properties mentioned above. To construct $G_m$, let us start by taking all the subindices $t\in\Gamma\setminus\Gamma_{m-1}$ with $m_t=m$. In case there are no $t$ with these properties we simply define $G_m:=G_{m-1}$. Thus, for each of those indexes with $m_t=m$ there exists a chain of cubes satisfying (\ref{Boman}) that connects $Q_t$ and $Q_a$ with adjacent cubes. However, we are going to consider just the first part of this chain which joins $Q_t$ with a cube $Q_s$ with $s\in\Gamma_{m-1}$. This element $s$ is the first one with this property (considering the order in the chain). Let us denote this portion of the original chain as $Q_{t,1},\cdots,Q_{t,r},Q_{t,r+1}$, with $Q_{t,1}=Q_t$ and $Q_{t,r+1}=Q_s$. Thus $Q_{t,r}$ is a cube such that its index does not belong to $\Gamma_{m-1}$. The number of cubes $r=r(t)$ is bounded by $r\leq \lambda^{2n}$. In order to prove this fact observe that all the cubes in the chain intersect each other in a set with Lebesgue measure zero thus 
\begin{eqnarray*}
\sum_{j=1}^{r}|Q_{t,j}|\leq \lambda^n |Q_{t,r}|.
 \end{eqnarray*}
$Q_{t,r}$ is a cube such that its index does not belong to $\Gamma_{m-1}$, thus using (h1) we know that $m(Q_{t,r})\geq m$ and the right hand side of the previous inequality satisfies that 
\begin{eqnarray*}
\lambda^n |Q_{t,r}|=\lambda^n 2^{-n\,m(Q_{t,r})} \leq \lambda^n 2^{-nm}.
\end{eqnarray*}
Now, $Q_t\subseteq \lambda Q_{t,j}$ for all $1\leq j\leq r$. Then using that $m_t=m$ we have that 
\begin{eqnarray*}
 r\lambda^{-n} 2^{-nm} \leq \sum_{j=1}^{r}|Q_{t,j}|.
\end{eqnarray*}
Thus, $r\leq \lambda^{2n}$. 

Now, we define an auxiliary graph $G=(V,E)$, where the set $V$ has a vertex $v_\ast\not\in \Gamma$. The rest of the vertices are the indexes in $\Gamma\setminus\Gamma_m$ of the cubes in $Q_{t,1},\cdots,Q_{t,r} $ for all $Q_t$ with $m_t=m$. Regarding the set $E$, we join two vertices in $V$ by an edge if they are the indices of two consecutive cubes in a chain $Q_{t,1},\cdots,Q_{t,r} $, or one is $v_\ast$ and the other one is the index of the tail cube $Q_{t,r}$ in a chain $Q_{t,1},\cdots,Q_{t,r} $. Next, using Lemma \ref{directed graph} we know that removing some edges from $G$ it is possible to obtain a rooted tree $\tilde{G}=(V,\tilde{E})$ with root $v_\ast$ such that the length of each chain connecting the vertices with $v_\ast$ does not exceed $\lambda^{2n}$. Finally, in order to construct $\Gamma_m$, we cut the subtrees added to the artificial vertex $v_\ast$ and add them to $\Gamma_{m-1}$, specifically to the indexes of the cubes $Q_{t,r+1}$ in the tail of chain. This procedure defines a rooted tree $G_m=(\Gamma_m,E_m)$ with root $a$ that containes $G_{m-1}$ as a subgraph. Once we have defined $G_m=(\Gamma_m,E_m)$, it remains to prove that $G_m$ satisfies (h1) to (h4).
 
 Property (h1) follows by construction. 
 
 Next, in order to prove (h2) it is sufficient to consider the case when $s$ belongs to $\Gamma_m \setminus \Gamma_{m-1}$. By construction $\lambda Q_s$ contains a cube $Q_t$ with $m_t=m$. Thus, $\lambda\,\text{diam}(Q_s)\geq \text{diam}(Q_t)$, then after some straightforward calculations we obtain $m_s\leq m+\log_2(\lambda)$.
 
The third condition (h3) also follows by construction. We only have to show the validity of (h4) in $\Gamma_m$. For this last case, we use the inductive hypothesis (h1)-(h4) on $\Gamma_{m-1}$, and the already proved (h1)-(h3) on $\Gamma_m$.
Now, given $t,t'\in\Gamma_m$ with $t\preceq t'$ we have to show that (\ref{QH distance}) holds. We may assume that $t'$ belongs to $\Gamma_m\setminus\Gamma_{m-1}$, otherwise (\ref{QH distance}) follows by using the inductive hypothesis. Thus, $m_{t'}\geq m$. We split the proof in two cases,
$m_t\geq m$ and $m_t\leq m-1$. Let us start with the first one. 

$m_t\geq m$: If $t$ belongs to $\Gamma_{m-j}$, for some $0\leq j\leq m+1$, then $j\leq 1+\log_2(\lambda)$. Indeed, using (h2) 
\begin{eqnarray*}
 m \leq m_t \leq m-j+1+\log_2(\lambda).
\end{eqnarray*}
Moreover, let $i$ be a number in the interval $m-j\leq i\leq m$. Hence, using (h3) we can conclude that the number of indexes $s\in\Gamma_{i}\setminus\Gamma_{i-1}$ such that $t\preceq s\preceq t'$ does not exceed $1+\lambda^{2n}$. Thus, 
\begin{eqnarray}\label{bound for k}
k(t,t')\leq (1+\lambda^{2n})(2+\log_2(\lambda)).
\end{eqnarray}
Now, using (h1) and (h2) we have 
\begin{eqnarray*}
m_{t'}-m_t\geq m-m_t \geq m-m-1-\log_2(\lambda)=-1-\log_2(\lambda).
\end{eqnarray*}
Thus, using (\ref{bound for k})
\begin{eqnarray*}
k(t,t')\leq k_\ast-l_\ast(1+\log_2(\lambda))\leq l_{\ast}(m_{t'}-m_t)+k_{\ast}.
\end{eqnarray*}

$m_t\leq m-1$: We know that $m_{t'}\geq m$, thus there exist two consecutive vertices $t_1, t_2$ in $t\preceq t_1\prec t_2\preceq t'$ such that $m_{t_1}\leq m-1$ and $m_{t_2}\geq m$. Now, $t_2$ and $t'$ are in the previous situation so we use (\ref{bound for k}) obtaining 
\begin{eqnarray*}
k(t_2,t')\leq (1+\lambda^{2n})(2+\log_2(\lambda)).
\end{eqnarray*}
Note that from (h1) we have that $t$ and $t_1$ belong to $\Gamma_{m-1}$, then using the inductive hypothesis 
\begin{eqnarray*}
k(t,t')&=&k(t,t_1)+k(t_1,t_2)+k(t_2,t')\\
&\leq & l_{\ast}(m_{t_1}-m_t)+k_{\ast}+1+(1+\lambda^{2n})(2+\log_2(\lambda))\\
&=&l_{\ast}(m_{t_1}-m_t)+k_{\ast}+l_\ast\leq l_{\ast}(m_{t'}-m_t)+k_{\ast},
\end{eqnarray*}
concluding the proof.
\end{proof}

\begin{proof}[Proof of Lemma \ref{Existence of a tree}] This result is a corollary of Lemma \ref{Length of the chain}. Indeed, 
given $s,t\in\Gamma$ with $t\preceq s$, we denote by $\alpha_{t}$ the center of $Q_{t}$ and take an arbitrary $y\in Q_s$. Then, using that two adjacent vertices in $\Gamma$ are the indexes of neighbor cubes, we have 
\begin{eqnarray*}
\text{dist}(\alpha_{t},y)&\leq& \sum_{t\preceq t'\preceq s} \text{diam}(Q_{t'})\\
&=& \sqrt{n}\sum_{t\preceq t'\preceq s} 2^{-m_{t'}}\\
&=& \sqrt{n}\,2^{-m_t}\sum_{t\preceq t'\preceq s} 2^{-(m_{t'}-m_t)}=(I)
\end{eqnarray*}
Next, from (\ref{QH distance}) 
\begin{eqnarray*}
(I)&\leq&\sqrt{n}\,2^{-m_t}\sum_{t\preceq t'\preceq s} 2^{-\frac{1}{l_\ast}(k(t,t')-k_\ast)}\\
&=&\sqrt{n}\,2^{-m_t}2^{k_\ast/l_\ast}\sum_{i=0}^{k(t,s)} \left(2^{-1/l_\ast}\right)^{i}.
\end{eqnarray*}
Finally, the following constant fulfills (\ref{Boman tree})
\begin{eqnarray*}
K:=2^{1+k_\ast/l_\ast}\sqrt{n}\sum_{i=0}^{\infty} \left(2^{-1/l_\ast}\right)^{i}.
\end{eqnarray*}

\end{proof}

%\section*{Acknowledgements}
%I would like to thank Professor David Jerison from Massachusetts Institute of Technology for the fruitful discussions we had about different aspects of Poincar\'e inequality during my visit some years ago. These discussions have been an inspiration for this work.

\end{appendices}

\end{document}